\documentclass[11pt,amssymb]{amsart}

\usepackage[OT2,T1]{fontenc}
\DeclareSymbolFont{cyrletters}{OT2}{wncyr}{m}{n}
\newcommand{\Sha}{\text{{\brus
SH}}}
\usepackage[mathscr]{eucal}

\usepackage{graphicx}
\usepackage{amssymb}
\usepackage{amsmath}
\usepackage{color}
\usepackage{amsmath,amscd}
\usepackage{pictexwd,dcpic}
\newtheorem{theorem}{Theorem}
\newtheorem{lemma}[theorem]{Lemma}
\newtheorem{proposition}[theorem]{Proposition}

\newcommand{\Gal}{\mbox{Gal}}

\newcommand\bZ{\mathbb Z}

\font\brus=wncyr10.240pk scaled 1200 .240pk
\begin{document}
\title[The Multinorm Principle]{The Multinorm Principle for Linearly Disjoint Galois Extensions}
\author[Pollio]{Timothy P. Pollio}
\author[Rapinchuk]{Andrei S. Rapinchuk}

\begin{abstract}
Let $L_1$ and $L_2$ be finite separable extensions of a global field
$K$, and let $E_i$ be the Galois closure of $L_i$ over $K$ for $i =
1, 2$. We establish a local-global principle for the product of
norms from $L_1$ and $L_2$ (so-called {\it multinorm principle})
provided that the extensions $E_1$ and $E_2$ are linearly disjoint
over $K$.
\end{abstract}

\date{\tt Version of \today}

\maketitle

\section{Introduction}
Let $K$ be a global field.  Given a finite extension $L/K$, we let
$J_K$ and $J_L$ denote the groups of ideles of $K$ and $L$
respectively, and let $N_{L/K}:J_L \rightarrow J_K$ denote the
natural extension of the norm map associated with $L/K$ (cf.
\cite[p.~73-75]{Cass}).  Then the extension $L/K$ is said to satisfy
the \textit{Hasse norm principle} if
\[ K^\times \cap N_{L/K}(J_L) = N_{L/K}(L^\times).  \]
The classical result of Hasse states that this is always the case if
$L/K$ is a cyclic Galois extension.  For general extensions (even
Galois extensions), the Hasse principle does not necessarily hold,
and its investigation has received a lot of attention.  The
obstruction to the Hasse principle is given by the quotient
\[\Sha(L/K)=\frac{K^\times \cap N_{L/K}(J_{L})}{N_{L/K}(L^\times)}\]
which is a finite group called the \textit{Tate-Shafarevich group}
of the extension $L/K$. (We note that it coincides with the
Tate-Shafarevich group of the corresponding norm torus
$\mbox{R}^{(1)}_{L/K}(\mbox{GL}_1)$, cf. \cite{Vos}, \S 11).

In \cite[p.~198]{Cass}, Tate gave the following cohomological
computation of $\Sha(L/K)$ for a Galois extension $L/K$: {\it Let $G
= \mathrm{Gal}(L/K)$, and for a valuation $v$ of $K,$ let $G^v$ be
the decomposition group of (a fixed extension of) $v$. Then
$\Sha(L/K)$ is the dual of (hence is isomorphic to) the kernel of
the map $H^3(G , \bZ) \to \prod_v H^3(G^v , \bZ)$ induced by
restriction.} Various aspects of the Hasse principle were
investigated in \cite{Gurak}, \cite{PlD1}, and \cite{PlD2}, and a
computation of $\Sha(L/K)$ for an arbitrary finite extension $L/K$
in terms of so-called representation groups of the relevant Galois
groups was given by Drakokhrust \cite{Drak}.

In \cite{Hurlimann}, H\"urlimann considered the tori of norm type
associated with a pair of finite extensions $L_1  ,  L_2$ of a
global field $K.$ The triviality of the Tate-Shafarevich group for
this torus is equivalent to the fact that
\begin{equation}\tag{M}
  K^\times \cap N_{L_1/K}(J_{L_1})N_{L_2/K}(J_{L_2})=N_{L_1/K}(L_1^\times)N_{L_2/K}(L_2^\times).
  \label{M}
\end{equation}
Following \cite{PlR}, we say that the pair $L_1  ,  L_2$ satsifies
the {\it multinorm principle} if (\ref{M}) holds. It was shown in
\cite{Hurlimann} that this is indeed the case if $L_1$ is a cyclic
Galois extension of $K$ and $L_2$ is an arbitrary Galois extension
(a similar result was independently obtained by Colliot-Th\'el\`ene
and Sansuc \cite{CTS}).  A more general sufficient condition for the
multinorm principle was given in \cite{PlR}, Proposition 6.11.  This
result was used to give a simplified proof of the Hasse principle
for Galois cohomology of simply connected outer forms of type $A_n$
over number fields (cf. \cite{PlR}, Ch. VI) and in the analysis of
the Margulis-Platonov conjecture for anisotropic inner forms of type
$A_n$ ({\it loc. cit.}, \S 9.2); it was also employed in \cite{PR2}
in the computation of the metaplectic kernel. More recently, another
sufficient condition for the multinorm principle was given in
\cite{PR} (cf. Proposition 4.2) in order to study the local-global
principle for embedding fields with an involutive automorphism into
simple algebras with involution; some further applications of this
result can be found in \cite{Gille}.

It should be emphasized that in {\it all} of these results it was
assumed that one of the extensions satisfies the Hasse principle. In
this light, the main result of this note looks quite surprising: we
show that no assumption of this nature is actually needed.

\vskip2mm

\noindent \textbf{Theorem.}\: \textit{Let $L_1$ and $L_2$ be two
finite separable extensions of a global field $K$, and let $E_i$ be
the Galois closure of $L_i$ over $K$ for $i = 1, 2$. If $E_1 \cap
E_2 = K$ (i.e., $E_1$ and $E_2$ are linearly disjoint over $K$) then
the pair $L_1 , L_2$ satisfies the multinorm principle.}

\vskip2mm

We notice that the conclusion of the theorem can be false for
non-linearly disjoint extensions. For example, if $L_1 = L_2 =: L$,
then the multinorm principle is equivalent to the norm principle for
$L/K$, hence may fail. See \S 4 for more sophisticated examples and
a discussion of a more general conjecture.

\vskip1mm

The proof of the theorem is based on the following sufficient
condition for the multinorm principle.

\begin{proposition}\label{P:1}
Let $L_1$ and $L_2$ be two finite separable extensions of $K$ such
that their Galois closures $E_1$ and $E_2$ satisfy $E_1 \cap E_2 =
K.$ Set $L = L_1L_2.$ If the map
\[ \phi: \Sha(L/K) \rightarrow \Sha(L_1/K) \times \Sha(L_2/K) \]
induced by the diagonal embedding $K^\times \hookrightarrow K^\times
\times K^\times$ is surjective, then the pair  $L_1  ,  L_2$
satisfies the multinorm principle.
\end{proposition}

In \S 2, we prove the proposition and also reduce the proof of the
theorem to the case where both $L_1$ and $L_2$ are Galois extensions
of $K$. Then, to complete the proof of the theorem, we verify that
the map $\phi$ is in fact surjective for any two linearly disjoint
Galois extensions  - cf. Proposition 3 in \S 3. Finally, \S 4
contains some additional results and examples related to the
multinorm principle.

\section{Proof of Proposition \ref{P:1}}\label{S:P1}

The following statement will enable us to prove Proposition
\ref{P:1}, but is also  of independent interest.
\begin{proposition}\label{P:2}
Let $L_1$ and $L_2$ be finite extensions of $K$ such that their
Galois closures $E_1$ and $E_2$ satisfy $E_1 \cap E_2=K$. Let
$L=L_1L_2$, and let
\[
S=K^\times\cap N_{L/K}(J_L) \ \ \text{and} \ \
T=N_{L_1/K}(L_1^\times)N_{L_2/K}(L_2^\times).
\]
Then the following conditions are equivalent:

\vskip2mm

{\rm (1)} The pair $L_1 , L_2$ satisfies the multinorm principle;

\vskip1mm

{\rm (2)} $K^{\times} \cap N_{L_i/K}(J_{L_i}) \subset T$ for $i = 1$
and $2$;

\vskip1mm

{\rm (3)} $K^{\times} \cap N_{L_i/K}(J_{L_i}) \subset T$ for at
least one index $i \in \{1 , 2\}$;

\vskip1mm

{\rm (4)} $S \subset T.$
\end{proposition}

\vskip2mm

\noindent \textit{Proof.} The implications $(1) \Rightarrow (2)
\Rightarrow (3) \Rightarrow (4)$ are obvious, while the nontrivial
implication $(4) \Rightarrow (1)$ is a consequence of the following
statement which is extracted from the the proof of Proposition 6.11
in \cite{PlR}.

\vskip2mm

\begin{lemma}\label{L:1}
Let $L_1$ and $L_2$ be as in Proposition \ref{P:2}. Then in the
above notations we have
$$
K^{\times} \cap N_{L_1/K}(J_{L_1}) N_{L_2/K}(J_{L_2}) = ST.
$$
\end{lemma}
\begin{proof}
For completeness, we (succinctly) reproduce the argument given in
\cite{PlR}. Let $M_i$ be the maximal abelian
extension of $K$ contained in $L_i$ for $i=1,2$, and $M$ be the maximal abelian
extension of $K$ contained in $L.$ Then by Galois theory the fact
that $E_1 \cap E_2 = K$ implies that

\vskip2mm

\noindent $\bullet$ $M = M_1M_2$ and $\Gal(M/K)$ is naturally
isomorphic to $\Gal(M/M_1) \times \Gal(M/M_2);$

\vskip1mm

\noindent $\bullet$ the maximal abelian extension of $L_i$ contained
in $L$ is $L_i M_{3-i}$ for $i = 1, 2.$

\vskip2mm

\noindent The crucial observation is that the map
$$
\varphi \colon J_{L_1}/L^{\times}_1 N_{L/L_1}(J_L) \times
J_{L_2}/L^{\times}_2 N_{L/L_2}(J_L) \longrightarrow J_{K}/K^{\times}
N_{L/K}(J_L),
$$
induced by the product of the norm maps $N_{L_1/K}$ and $N_{L_2/K},$
is an isomorphism, which is proved by showing that $\varphi$ is
surjective and that its domain and target have the same order. To
this end, we consider the following commutative diagram
\begin{equation}\label{E:2}
\begin{array}{ccc}
J_{M_1}/M^{\times}_1N_{M/M_1}(J_M) \times
J_{M_2}/M^{\times}_2N_{M/M_2}(J_M) &
\stackrel{\psi}{\longrightarrow} & J_K/K^{\times} N_{M/K}(J_M) \\
\theta_1 \times \theta_2 \downarrow & & \downarrow \theta \\
\Gal(M/M_1) \times \Gal(M/M_2) & \stackrel{\iota}{\longrightarrow} &
\Gal(M/K)
\end{array},
\end{equation}
where $\psi$ is constructed analogously to $\varphi,$
$$
\theta_i \colon J_{M_i}/M^{\times}_iN_{M/M_i}(J_M) \to \Gal(M/M_i) \
\ \text{and} \ \ \theta \colon J_K/K^{\times} N_{M/K}(J_M) \to
\Gal(M/K)
$$
are the isomorphisms given by the corresponding Artin maps (cf.
\cite[Ch. VII]{Cass}), and $\iota$ is induced by the canonical
embeddings $\Gal(M/M_i) \to \Gal(M/K);$ the commutativity of
(\ref{E:2}) follows from Proposition 4.3 in \cite{Cass}. In our
situation, $\iota$ is an isomorphism, so $\psi$ is also an
isomorphism, implying that
\begin{equation}\label{E:3}
J_K = K^{\times} N_{M_1/K}(J_{M_1}) N_{M_2/K}(J_{M_2}).
\end{equation}
We now recall the fact that for any finite separable extension $P/F$
of global fields we have
$$
F^{\times} N_{P/F}(J_P) = F^{\times} N_{R/F}(J_R),
$$
where $R$ is the maximal abelian extension of $F$ contained in $P$
(cf. \cite[Exercise 8]{Cass}). Thus,
$$
K^{\times} N_{L_i/K}(J_{L_i}) = K^{\times}N_{M_i/K}(J_{M_i}) \ \
\text{for} \ \ i = 1, 2,
$$
which in conjunction with (\ref{E:3}) yields that
$$
J_K = K^{\times} N_{L_1/K}(J_{L_1}) N_{L_2/K}(J_{L_2}),
$$
proving that $\varphi$ is surjective. On the other hand, since
$L_1M_2$ is the maximal abelian extension of $L_1$ contained in $L,$
using the fundamental isomorphism of global class field theory we
obtain
\begin{eqnarray*}
\vert J_{L_1} /L^{\times}_1 N_{L/L_1}(J_{L}) \vert = \vert J_{L_1}
/L^{\times}_1 N_{L_1M_2/L_1}(J_{L_1M_2}) \vert = [L_1M_2 : L_1] = \\ =
[M_2 : K] = [M : M_1] = \vert J_{M_1}/M^{\times}_1 N_{M/M_1}(J_M)
\vert,
\end{eqnarray*}
and similarly
$$
\vert J_{L_2}/L^{\times}_2 N_{L/L_2}(J_L) \vert = \vert J_{M_2}
/M^{\times}_2 N_{M/M_2}(J_M) \vert \ \ \text{and} \ \ \vert
J_K/K^{\times} N_{L/K}(J_L) \vert = \vert J_K/K^{\times}
N_{M/K}(J_M) \vert.
$$
Since $\psi$ is an isomorphism, these equation imply that the domain
and the target of $\varphi$ have the same order, proving  that
$\varphi$ is in fact an isomorphism.

Now, take any $a \in K^{\times} \cap N_{L_1/K}(J_{L_1})
N_{L_2/K}(J_{L_2}),$ and write it in the form
$$
a = N_{L_1/K}(x_1) N_{L_2/K}(x_2) \ \ \text{with} \ \ x_i \in
J_{L_i}.
$$
Then $(x_1L_1^{\times}N_{L/L_1}(J_{L}) , x_2L^{\times}_2
N_{L/L_2}(J_L)) \in \mathrm{Ker}\: \varphi.$ Using the injectivity
of $\varphi$ established above, we see that we can write
$$
x_i = y_i N_{L/L_i}(z_i) \ \ \text{with} \ \ y_i \in L^{\times}_i, \
z_i \in J_L \ \ \text{for} \ \ i = 1, 2.
$$
Then
$$
a = (N_{L_1/K}(y_1)N_{L_2/K}(y_2)) N_{L/K}(z_1z_2) \in TS.
$$
This proves the inclusion
$$
K^{\times} \cap N_{L_1/K}(J_{L_1})N_{L_2/K}(J_{L_2}) \subset ST,
$$
while the reverse inclusion is obvious.
\end{proof}

\vskip2mm

\noindent {\bf Remark.} If one of the  $L_i$'s satisfies the
usual Hasse norm principle then condition (3) of Proposition
\ref{P:2} obviously holds for this $i$. This yields the multinorm
principle in this situation, which is precisely the assertion of
Proposition 6.11 in \cite{PlR}. Thus, the latter is a particular
case of our Proposition \ref{P:2}.

\vskip2mm

Before proceeding with the proof of Proposition \ref{P:1}, we will
now use Lemma \ref{L:1} to give

\vskip1mm

\noindent {\it Reduction of the theorem to the Galois case.} Let
$L_1 , L_2$ be as in the theorem, and let us assume that we already
know that their Galois closures $E_1 , E_2$ satisfy the multinorm
principle. We will now show that the pair $L_1 , L_2$  satisfies the
multinorm principle as well. Generalizing the notions introduced in
the proof of Proposition \ref{P:2}, for a pair of finite extensions
$P_1$ and $P_2$ of $K$, we set
$$
S_{P_1 , P_2} = K^{\times} \cap N_{P_1P_2/K}(J_{P_1P_2}) \ \
\text{and} \ \ T_{P_1 , P_2} = N_{P_1/K}(P^{\times}_1)
N_{P_2/K}(P^{\times}_2).
$$
We also set
$$
R_{P_1 , P_2} = K^{\times} \cap
N_{P_1/K}(J_{P_1})N_{P_2/K}(J_{P_2}).
$$
We note that for any other finite extensions $P'_1$ and $P'_2$ of
$K$ we have the inclusions
\begin{equation}\label{E:inclusion}
S_{P_1 , P_2} \subset R_{P'_1 , P_2} \ \ \text{and} \ \ S_{P_1 ,
P_2} \subset R_{P_1 , P'_2}.
\end{equation}
Now, applying Lemma \ref{L:1} twice in conjunction with
\eqref{E:inclusion}, we obtain
\begin{equation}\label{E:inclusion1}
R_{L_1 , L_2} = T_{L_1 , L_2} S_{L_1 , L_2} \subset T_{L_1 , L_2}
R_{E_1 , L_2} = T_{L_1 , L_2} T_{E_1 , L_2} S_{E_1 , L_2} \subset
T_{L_1 , L_2} T_{E_1 , L_2} R_{E_1 , E_2}.
\end{equation}
Since by our assumption the multinorm principle holds for the pair
$E_1 , E_2$, we have $R_{E_1 , E_2}  = T_{E_1 , E_2}$, so
\eqref{E:inclusion1} becomes
$$
R_{L_1 , L_2} \subset T_{L_1 , L_2}T_{E_1 , L_2} T_{E_1 , E_2} =
T_{L_1 , L_2},
$$
which means that the multinorm principle holds for the pair $L_1 ,
L_2$. \hfill $\Box$

\vskip2mm

To complete the proof of Proposition \ref{P:1}, we need  the
following elementary group-theoretic lemma.
\begin{lemma}\label{L:2}
Let $\mathcal{A}$ be an abelian group with subgroups $\mathcal{B}$
and $\mathcal{C}.$ Then the sequence
$$
\mathcal{A} \stackrel{f}{\longrightarrow}
\frac{\mathcal{A}}{\mathcal{B}} \times
\frac{\mathcal{A}}{\mathcal{C}} \stackrel{g}{\longrightarrow}
\frac{\mathcal{A}}{\mathcal{B}\mathcal{C}} \longrightarrow 1,
$$
where $f$ and $g$ are defined by
$$
f(x) = (x\mathcal{B} , x\mathcal{C}) \ \ \text{and} \ \
g(x\mathcal{B} , y\mathcal{C}) = xy^{-1}\mathcal{B}\mathcal{C},
$$
is exact.
\end{lemma}

\vskip2mm

\noindent \textit{Proof of Proposition 1.} Applying Lemma \ref{L:2}
to the group $\mathcal{A} = K^{\times} \cap N_{L/K}(J_L)$ and its
subgroups
$$
\mathcal{B} = N_{L_1/K}(L^{\times}_1) \cap N_{L/K}(J_L) \ \
\text{and} \ \ \mathcal{C} = N_{L_2/K}(L^{\times}_2) \cap
N_{L/K}(J_L),
$$
we obtain the following exact sequence
\begin{eqnarray}\label{E:1}
K^{\times} \cap N_{L/K}(J_L) \stackrel{f}{\longrightarrow}
\frac{K^{\times} \cap N_{L/K}(J_L)}{N_{L_1/K}(L^{\times}_1) \cap
N_{L/K}(J_L)} \times \frac{K^{\times} \cap
N_{L/K}(J_L)}{N_{L_2/K}(L^{\times}_2) \cap N_{L/K}(J_L)}
\stackrel{g}{\longrightarrow} \\ \longrightarrow \frac{K^{\times}
\cap N_{L/K}(J_L)}{(N_{L_1/K}(L^{\times}_1) \cap
N_{L/K}(J_L))(N_{L_2/K}(L^{\times}_2) \cap N_{L/K}(J_L))}
\longrightarrow 1. \nonumber
\end{eqnarray}
By our assumption, the composite homomorphism
\begin{eqnarray*}
\text{{\brus SH}}(L/K) = \frac{K^{\times} \cap
N_{L/K}(J_L)}{N_{L/K}(L^{\times})}
\stackrel{\bar{f}}{\longrightarrow} \frac{K^{\times} \cap
N_{L/K}(J_L)}{N_{L_1/K}(L^{\times}_1) \cap N_{L/K}(J_L)} \times
\frac{K^{\times} \cap N_{L/K}(J_L)}{N_{L_2/K}(L^{\times}_2) \cap
N_{L/K}(J_L)} \stackrel{h}{\longrightarrow} \\
\longrightarrow \frac{K^{\times} \cap
N_{L_1/K}(J_{L_1})}{N_{L_1/K}(L^{\times}_1)} \times \frac{K^{\times}
\cap N_{L_2/K}(J_{L_2})}{N_{L_2/K}(L^{\times}_2)} = \text{{\brus
SH}}(L_1/K) \times \text{{\brus SH}}(L_2/K),
\end{eqnarray*}
where $\bar{f}$ is induced by $f$ and $h$ by the inclusions
$K^{\times} \cap N_{L/K}(J_L) \subset K^{\times} \cap
N_{L_i/K}(J_{L_i})$ for $i = 1, 2,$ is surjective. Since $h$ is
obviously injective, we conclude that $\bar{f},$ hence $f,$ is
surjective. So, the exact sequence (\ref{E:1}) yields that its third
term is trivial, i.e.
$$
S = K^{\times} \cap N_{L/K}(J_L) = (N_{L_1/K}(L^{\times}_1) \cap
N_{L/K}(J_L))(N_{L_2/K}(L^{\times}_2) \cap N_{L/K}(J_L)) \subset
$$
$$
\subset N_{L_1/K}(L^{\times}_1)N_{L_2/K}(L^{\times}_2) = T.
$$
This verifies condition (4) of Proposition 2, thereby
yielding the validity of the multinorm principle for the pair $L_1 ,
L_2.$ \hfill $\Box$

\section{Proof of the Main Theorem}\label{S:MT}

As we have seen in \S~\ref{S:P1}, it is enough to prove the main
theorem assuming that both $L_1$ and $L_2$ are Galois extensions of
$K$. In this case, the claim is a consequence of Proposition
\ref{P:1} combined with the following statement.
\begin{proposition}\label{P:3}
Let $L_1$ and $L_2$ be Galois extensions of $K$ with $L_1 \cap L_2
 = K$, and let $L = L_1L_2$.  Then the map
\[ \phi \colon \Sha(L/K) \rightarrow \Sha(L_1/K) \times \Sha(L_2/K) \] induced
by the diagonal embedding $K^{\times} \hookrightarrow
K^{\times} \times K^{\times}$ is surjective.
\end{proposition}

Our proof relies on properties of the deflation and residuation maps
for the Tate cohomology groups, introduced in \cite{Weiss} and
\cite{Horie},  and their interaction with the fundamental
isomorphisms of class field theory. Since these maps are rarely
used, we briefly recall in the appendix their construction, which is
needed to prove the key Lemma \ref{L:5}.

Given a finite group $G$ and a $G$-module $A$, we let $\hat{H}^i(G ,
A)$ denote the $i$th Tate cohomology group (cf., for example,
\cite[Ch. IV, \S~6]{Cass}). For a normal subgroup $H$ of $G$ and any
$i \geqslant 0$, one can define the {\it deflation map}
$$
\mathrm{Def}^G_{G/H} \colon \hat{H}^{-i}(G , A) \to \hat{H}^{-i}(G/H ,
A^H).
$$
The deflation map is natural; in particular, it has the following
properties.
\begin{lemma}\label{L:3}
For any $G$-module homomorphism $f \colon A \rightarrow B$ and any
$i \geqslant 0$, the diagram
\[
\begin{CD}
\hat{H}^{-i}(G , A) @>>> \hat{H}^{-i}(G , B) \\
 @VV\mathrm{Def}^G_{G/H}V @VV\mathrm{Def}^G_{G/H}V \\
 \hat{H}^{-i}(G/H , A^H) @>>> \hat{H}^{-i}(G/H , B^H)
\end{CD}
\]
in which the horizontal maps are induced by $f$, is commutative.
\end{lemma}
\begin{proof}
This is Proposition 8 in \cite{Weiss}.
\end{proof}

\begin{lemma}\label{L:4}
Let
\begin{equation}\label{E:ES1}
0 \rightarrow A \rightarrow B \rightarrow C \rightarrow 0
\end{equation}
be an exact sequence of $G$-modules, and assume that the induced
sequence of $G/H$-modules
\begin{equation}\label{E:ES2}
0 \rightarrow A^H \rightarrow B^H \rightarrow C^H \rightarrow 0
\end{equation}
is also exact. Then for any $i \geqslant 1$ the diagram
\[
\begin{CD}
 \hat{H}^{-i}(G , C) @>>> \hat{H}^{-i+1}(G , A)  \\
 @VV\mathrm{Def}^G_{G/H}V @VV\mathrm{Def}^G_{G/H}V  \\
 \hat{H}^{-i}(G/H,C^H) @>>> \hat{H}^{-i+1}(G/H,A^H)
\end{CD}
\]
in which the horizontal maps are the coboundary maps arising from
the exact sequences \eqref{E:ES1} and \eqref{E:ES2}, is commutative.
\end{lemma}
\begin{proof}
This is Proposition 4 in \cite{Weiss}.
\end{proof}

Our proof also makes use of the {\it residuation map}
$\mathrm{Rsd}^G_H$ -- see the appendix. The key property that we
need is that in the case of interest to us, the residuation map is
the dual of the usual inflation map. More precisely, we have the
following.
\begin{lemma}\label{L:5}
Let $G=H \times K$ and identify $G/K$ with $H$.  Then for $i \geq 2$
the residuation and inflation maps in the following diagram
\[
\begindc{\commdiag}[5]
\obj(0,10)[a]{$\hat{H}^{-i}(G,\bZ)$}
\obj(14,10)[b]{$\hat{H}^{i}(G,\bZ)$}
\obj(0,0)[c]{$\hat{H}^{-i}(H,\bZ)$}
\obj(14,0)[d]{$\hat{H}^{i}(H,\bZ)$} \obj(7,0){$\times$}
\obj(7,10){$\times$} \obj(28,10)[e]{$\hat{H}^{0}(G,\bZ)$}
\obj(28,0)[f]{$\hat{H}^{0}(H,\bZ)$} \mor{a}{c}{\rm $\mathrm{Rsd}^G_H$}
\mor{d}{b}{\rm $\mathrm{Inf}^G_H$} \mor{f}{e}{\rm $\mathrm{Cor}^G_H$} \mor{b}{e}{$\cup$}
\mor{d}{f}{$\cup$}
\enddc
\]
are adjoint with respect to the pairings given by the
$\cup$-products. That is,
\[f \cup \mathrm{Inf}^G_H(\psi) = \mathrm{Cor}^G_H(\mathrm{Rsd}^G_H(f) \cup \psi)\]
for every $f \in \hat{H}^{-i}(G,\bZ)$ and $\psi \in \hat{H}^{i}(H,\bZ)$.
\end{lemma}
\vskip2mm

\begin{proof} This uses an explicit construction of the residuation map
and will be given in the appendix. \end{proof}

\vskip2mm

Another critical ingredient of the proof of Proposition \ref{P:3} is
the following result of K.~Horie and M.~Horrie \cite{Horie} that
shows how the deflation and residuation maps interact with the
isomorphisms from class field theory. For a global field $K$, we let
$C_K = J_K/K^{\times}$ denote the idele class group. Furthermore,
given a Galois extension $F/K$ of global fields, for any
$\mathrm{Gal}(F/K)$-module $A$ we write $\hat{H}(F/K , A)$ instead
of $\hat{H}(\mathrm{Gal}(F/K) , A)$, and then for any $i \in
\mathbb{Z}$ there is a canonical isomorphism $\Phi_F \colon
\hat{H}^{i-2}(F/K , \mathbb{Z}) \to \hat{H}^i(F/K , C_F)$ called the
{\it Tate isomorphism} (cf. \cite{Cass}, Ch. VII).
\begin{lemma}\label{L:6}
{\rm (\cite{Horie}, Theorem 1)} Let $E \subset F$ be Galois
extensions of a global field $K$. Then for any $i \geqslant 0$, the
following diagram
\begin{equation}\label{E:Tate}
\begin{CD}
\hat{H}^{-i-2}(F/K , \bZ) @>\Phi_F>>\hat{H}^{-i}(F/K , C_F)  \\
 @VV\mathrm{Rsd}^{\mathrm{Gal}(F/K)}_{\mathrm{Gal}(E/K)} V @VV\mathrm{Def}^{\mathrm{Gal}(F/K)}_{\mathrm{Gal}(E/K)}V  \\
\hat{H}^{-i-2}(E/K , \bZ) @>\Phi_E>> \hat{H}^{-i}(E/K , C_E)
\end{CD}
\end{equation}
commutes.
\end{lemma}
(We will only use this lemma for $i = 1$.)

\vskip5mm

\noindent {\it Proof of Proposition \ref{P:3}.} For a finite Galois
extension $F/K$, we let
$$
\kappa_F \colon \hat{H}^0(F/K , F^{\times}) \to \hat{H}^0(F/K , J_F)
$$
denote the map induced by the inclusion $F^{\times}  \to J_F$. Then
clearly $\Sha(F/K) = \mathrm{Ker}\: \kappa_F$. Now, let $G_j =
\mathrm{Gal}(L_j/K)$ for $j = 1, 2$. Since $L_1$ and $L_2$ are
assumed to be linearly disjoint, for $L = L_1L_2$ and $G =
\mathrm{Gal}(L/K)$ there is a natural isomorphism
$$
G = G_1 \times G_2,
$$
which in particular allows us to identify $G/G_{3-j}$ with $G_j$ for
$j = 1, 2$. Considering the inclusion $L^{\times} \to J_L$ as part
of the exact sequence of $G$-modules $1 \to L^{\times} \to J_L \to
C_L \to 1$ and applying Lemmas \ref{L:3} and \ref{L:4} to $H =
G_{3-j}$ with $i = 1$ we obtain (observing that the corresponding
sequence \eqref{E:ES2} is $1 \to L^{\times}_j \to J_{L_j} \to
C_{L_j} \to 1$, cf. \cite[Ch. VII, Prop. 8.1]{Cass}) the following
commutative diagram with exact rows:
\begin{equation}\label{E:diagr1}
\begin{CD}
\hat{H}^{-1}(G,C_L) @>>>\hat{H}^{0}(G,L^\times) @>\kappa_L>>\hat{H}^{0}(G,J_L)  \\
 @VV\mbox{Def}^G_{G_j}V @VV\mbox{Def}^G_{G_j}V @VV\mbox{Def}^G_{G_j}V \\
 \hat{H}^{-1}(G_j,C_{L_j}) @>>> \hat{H}^{0}(G_j,L_j^\times) @>\kappa_{L_j}>> \hat{H}^{0}(G_j,J_{L_j})
\end{CD}
\end{equation}
for each $j = 1, 2$. Since the deflation map in dimension $0$ is
induced by the identity map (cf. the appendix), we see that the map
$\phi$ in Proposition \ref{P:3} is the map $\mathrm{Ker}(\kappa_L)
\rightarrow \mathrm{Ker}(\kappa_{L_1}) \times \mathrm{Ker}(\kappa_{L_2})$
induced by $\mathrm{Def}^G_{G_1} \times \mathrm{Def}^G_{G_2}$. So,
it follows from \eqref{E:diagr1} that $\phi$ is surjective if
\begin{equation}\label{E:Defl}
\mathrm{Def}^G_{G_1} \times \mathrm{Def}^G_{G_2} \colon
\hat{H}^{-1}(G , C_L) \to \hat{H}^{-1}(G_1 , C_{L_1}) \times
\hat{H}^{-1}(G_2 , C_{L_2})
\end{equation}
is such. Now, using Lemma \ref{L:6} with $i = 1$, we obtain the
following commutative diagram
\[
\begin{CD}
\hat{H}^{-3}(G,\bZ) @>\Phi_L>> \hat{H}^{-1}(G,C_L) \\
@VV\mathrm{Rsd}^G_{G_j}V   @VV\mathrm{Def}^G_{G_j}V\\
\hat{H}^{-3}(G_j,\bZ) @>\Phi_{L_j}>> \hat{H}^{-1}(G_j,C_{L_j})
\end{CD}
\]
for each $j = 1, 2$. So, the surjectivity of \eqref{E:Defl} is
equivalent to that of
\begin{equation}\label{E:Rsd}
\mathrm{Rsd}^G_{G_1} \times \mathrm{Rsd}^G_{G_2} \colon
\hat{H}^{-3}(G , \bZ) \to \hat{H}^{-3}(G_1 , \bZ) \times
\hat{H}^{-3}(G_2 , \bZ).
\end{equation}
For this, we will use the duality between the residuation and
inflation maps provided by Lemma \ref{L:5}. More precisely, it is
well-known (cf., for example, \cite[Theorem 6.6, p.~250]{CE}) that
for any finite group $H$ and any $i \in \bZ$, the $\cup$-product
defines a perfect pairing
$$
\alpha_H \colon \hat{H}^{-i}(H , \bZ) \times \hat{H}^{i}(H , \bZ)
\to \hat{H}^0(H , \bZ) = \bZ/\vert H \vert\bZ.
$$
On the other hand, in our situation, $\mathrm{Cor}^G_{G_j}$
identifies $H^0(G_j , \bZ) = \bZ/\vert G_j \vert\bZ$ with
$$
\vert G_{3-j}\vert\bZ /\vert G \vert \bZ \subset \bZ /\vert G \vert
\bZ = \hat{H}^0(G , \bZ).
$$
It follows that $\alpha = \mathrm{Cor}^G_{G_1} \circ \alpha_{G_1} +
\mathrm{Cor}^G_{G_2} \circ \alpha_{G_2}$ defines a perfect pairing
$$
(\hat{H}^{-i}(G_1 , \bZ) \times \hat{H}^{-i}(G_2 , \bZ)) \times
(\hat{H}^i(G_1 , \bZ) \times \hat{H}^i(G_2 , \bZ)) \to H^0(G , \bZ).
$$
Furthermore, by Lemma \ref{L:5}, we have the following commutative
diagram

\[
\begindc{\commdiag}[5]
\obj(0,0)[c]{$(\hat{H}^{-3}(G_1 , \bZ) \times \hat{H}^{-3}(G_2 , \bZ))$}
\obj(36,0)[d]{$(\hat{H}^3(G_1 , \bZ) \times \hat{H}^3(G_2 , \bZ))$}
\obj(0,10)[a]{$\hat{H}^{-3}(G , \bZ)$}
\obj(36,10)[b]{$H^3(G , \bZ)$} \obj(18,0){$\times$}
\obj(18,10){$\times$}
\obj(62,5)[e]{$\hat{H}^{0}(G,\bZ)$}
\mor{a}{c}{$\mathrm{Rsd}^G_{G_1} \times \mathrm{Rsd}^G_{G_2}$}[-1,0]
\mor{d}{b}{$\mathrm{Inf}^G_{G_1} + \mathrm{Inf}^G_{G_2} $} \mor{b}{e}{$\cup$}
\mor{d}{e}{$\alpha$}
\enddc
\]
Thus, the surjectivity of \eqref{E:Rsd} is equivalent to the
injectivity of $\mathrm{Inf}^G_{G_1} + \mathrm{Inf}^G_{G_2}$, and
the proof of the proposition is completed by the following
statement.
\begin{lemma}
For any finite group $G$ of the form $G = G_1 \times G_2$ and any $i
\geqslant 1$, the map
$$
\mathrm{Inf}^G_{G_1} + \mathrm{Inf}^G_{G_2} \colon \hat{H}^i(G_1 ,
\bZ) \times \hat{H}(G_2 , \bZ) \to \hat{H}^i(G , \bZ)
$$
is injective.
\end{lemma}
\begin{proof}
For a subgroup $H \subset G$, we let $\mathrm{Res}^G_H \colon
\hat{H}^i(G , \bZ) \to \hat{H}^i(H , \bZ)$ denote the corresponding
restriction map. Identifying $G/G_{3-j}$ with $G_j$ as above, it is
easy to see that the composition
$$
\mathrm{Res}^G_{G_j} \circ \mathrm{Inf}^G_{G_j} \colon
\hat{H}^i(G_j , \bZ) \to \hat{H}^i(G_j , \bZ)
$$
is the identity map, while the composition $\mathrm{Res}^G_{G_{3-j}}
\circ \mathrm{Inf}^G_{G_j}$ is zero, and our assertion follows.
\end{proof}

\vskip2mm

\noindent \textbf{Remark.} We note that the deflation map in the
context of Tate-Shafarevich groups and its connection with the
inflation map was used in \cite[p.~97]{Opolka} for a different
purpose.

\section{Examples and Extensions}

In this section we give examples where the multinorm principle fails
and prove some results that compliment and extend the main theorem.

\vskip2mm

\noindent \textbf{Example 1.} {\it For non-Galois extensions, the
condition $L_1 \cap L_2 = K$ may not imply the multinorm principle
for the pair $L_1 , L_2$.} Indeed, let $F/K$ be a Galois extension
with Galois group $G = \mathrm{Gal}(F/K)$ isomorphic to $A_6$ as in
Lemma 2 of \cite{PlD1}, and let $H$ be a subgroup of $G$ of index 10
(see {\it loc. cit.} or \cite{PlR}, p. 311). Since $A_6$ is simple,
we can choose $\sigma \in G$ such that $\sigma H \sigma^{-1} \neq
H$. Set
$$
L_1 = F^H \ \ \text{and} \ \ L_2 = F^{\sigma H \sigma^{-1}} =
\sigma(L_1).
$$
Clearly, $A_6$ does not have any subgroups of index $2$ or $5$, so
$\langle H \, , \, \sigma H \sigma^{-1} \rangle = G$ and therefore
\begin{equation}\label{E:intersection}
L_1 \cap L_2 = K.
\end{equation}
On the other hand, since $L_1$ and $L_2$ are Galois-conjugate over
$K$, we have
$$
N_{L_1/K}(L^{\times}_1) = N_{L_2/K}(L^{\times}_2) \ \ \text{and} \ \
N_{L_1/K}(J_{L_1}) = N_{L_2/K}(J_{L_2}).
$$
This means that the multinorm principle for the pair $L_1 , L_2$ is
equivalent to the Hasse norm principle for $L_1/K$. However,
according to Theorem 1 of \cite{PlD1}, the latter actually fails for
$L_1/K$. Thus, the pair $L_1 , L_2$ does not satisfy the Hasse norm
principle despite (\ref{E:intersection}). \hfill $\Box$

\vskip2mm

We note that the extensions $L_1$ and $L_2$ in Example 1 are not
linearly disjoint. However, even for linearly disjoint extensions
$L_1 , L_2$ their Galois closures $E_1$ and $E_2$ need not satisfy
$E_1 \cap E_2 = K$ (e.g. for the linearly disjoint extensions $L_1 =
\mathbb{Q}(\sqrt[3]{5})$ and $L_2 = \mathbb{Q}(\sqrt[3]{7})$ of
$\mathbb{Q}$, we have $E_1 \cap E_2 = \mathbb{Q}(\zeta_3)$ where
$\zeta_3$ is a primitive 3rd root of unity), which is required to
apply our Main Theorem. So, the question of whether any pair $L_1 ,
L_2$ of linearly disjoint extensions of $K$ satisfies the multinorm
principle remains open.

On the other hand, it would be interesting to analyze the multinorm
principle for at least pairs of Galois extensions $L_1 , L_2$ such
that $L_1 \cap L_2 \neq K$. This case is not well-understood as of
now, but the following proposition clarifies the nature of
additional conditions one needs to impose to avoid obvious
counter-examples.
\begin{proposition}
Let $L_1$ and $L_2$ be finite Galois extensions of $K$ satisfying
$L_1 \cap L_2 = K$, and let $L_3$ be any finite extension of $L_1$.
If $L_1/K$ fails to satisfy the norm principle, then the pair
$L_1L_2 , L_3$ fails to satisfy the multinorm principle.
\end{proposition}
\begin{proof}
It follows from Proposition \ref{P:3} that the natural homomorphism
\[ \Sha(L_1L_2/K) \rightarrow \Sha(L_1/K) \] is  surjective.
Since $\Sha(L_1/K)$ is non-trivial, this means that there exists $x
\in K^\times \cap N_{L_1L_2/K}(J_{L_1L_2})$ that is not in
$N_{L_1/K}(L_1^\times)$. Then $x$ lies in $K^\times \cap
N_{L_1L_2/K}(J_{L_1L_2})N_{L_3/K}(J_{L_3})$, but cannot be contained
in $N_{L_1L_2/K}((L_1L_2)^\times)N_{L_3/K}(L_3^\times) \subseteq
N_{L_1/K}(L_1^\times)$.
\end{proof}

\vskip3mm

Based on the (negative) result of the proposition, we would like to
propose the following.

\vskip2mm

\noindent \textbf{Conjecture.} {\it Let $L_1$ and $L_2$ be finite
Galois extensions of $K$. If every extension $P$ of $K$ contained in
$L_1 \cap L_2$ satisfies the norm principle then the pair $L_1 ,
L_2$ satisfies the multinorm principle. (It may be enough to require
that only the intersection $L_1 \cap L_2$ satisfies the norm
principle.)}

\vskip2mm

We note that, if proved, this conjecture would imply that a pair
$L_1 , L_2$ of finite Galois extensions of $K$ satisfies the
multinorm whenever the intersection $L_1 \cap L_2$ is a cyclic
extension of $K$.

\vskip3mm

Next, we would like to point out that in some simple cases the Main
Theorem can be proved without any use of group cohomology. The first
such instance is when both extensions are biquadratic.
\begin{proposition}
Let $L_1$ and $L_2$ be biquadratic extensions of $K$ satisfying $L_1
\cap L_2 = K$. Then the pair $L_1 , L_2$ satisfies the multinorm
principle.
\end{proposition}
\begin{proof}
Write $L_1 = K(\sqrt{a} , \sqrt{b})$ and $L_2 = K(\sqrt{c} ,
\sqrt{d})$. If at least one of the extensions satisfies the norm
principle then the result follows from Proposition 2 (see the remark
after the proposition). So, we only need to consider the case were
both extensions fail to satisfy the norm principle. Using Tate's
computation of the Tate-Shafarevich group for a Galois extension
mentioned in the introduction, one readily sees that all local
degrees of $L_i$ over $K$ are either 1 or 2, and then $\Sha(L_i/K)$
is of order 2 for both $i = 1, 2$. We let $S$ and $T$ denote the
sets of places of $K$ that split in $K(\sqrt{a})$ and $K(\sqrt{c})$
respectively. Following \cite[Exercise 5]{Cass}, consider the
following homomorphisms of $K^{\times}$ to $\{ \pm 1 \}$:
$$
\varphi_1(x) = \prod_{v \in S} (x , b)_v \ \ \text{and} \ \
\varphi_2(x) = \prod_{v \in T} (x , d)_v,
$$
where $(x , y)_v$ denotes the Hilbert symbol at $v$. Clearly $\ker
\varphi_i$ is an index two subgroup in $K^{\times}$ that according
to {\it loc. cit.} admits the following description
\begin{equation}\label{E:square}
\ker \varphi_i = \{ x \in K^{\times} \: \vert \: x^2 \in
N_{L_i/K}(L^{\times}_i) \}
\end{equation}
for $i = 1, 2$. Since $b$ and $d$ define different cosets modulo
${K^{\times}}^2$, it follows from properties of the Hibert symbol
(cf. \cite[Exercise 2.6]{Cass}) that the homomorphisms $\varphi_1$
and $\varphi_2$ are distinct, hence $(\ker \varphi_1)(\ker
\varphi_2) = K^{\times}$. Using (\ref{E:square}), we obtain the
inclusion
\begin{equation}\label{E:in}
{K^{\times}}^2 \subset N_{L_1/K}(L^{\times}_1)
N_{L_2/K}(L^{\times}_2).
\end{equation}
Now, let $x_i \in K^{\times}$ be such that $\varphi_i(x_i) = -1$.
Then $x^2_i \notin N_{L_i/K}(L^{\times}_i)$. On the other hand,
since  all the local degrees of $L_i$ over $K$ are either $1$ or
$2$, we see that $x^2_i \in K^{\times} \cap N_{L_i/K}(J_{L_i})$.
This means that the coset $x^2_i N_{L_i/K}(L^{\times}_i)$ is a
generator of $\Sha(L_i/K) \simeq \mathbb{Z}/2\mathbb{Z}$, hence
$$
K^{\times} \cap N_{L_i/K}(J_{L_i}) = \{1 ,
x^2_i\}N_{L_i/K}(L^{\times}_i).
$$
Now, taking into account (\ref{E:in}), we see that
$$
K^{\times} \cap N_{L_i/K}(L^{\times}_i) \subset
N_{L_1/K}(L^{\times}_1) N_{L_2/K}(L^{\times}_2),
$$
verifying thereby condition (2) of Proposition \ref{P:2} and
completing the proof of the multinorm principle for the pair $L_1 ,
L_2$.
\end{proof}

\vskip2mm

Another instance is when both extensions are of a prime degree $p$.
We recall that any extension $L/K$ of degree $p$ satisfies the norm
principle (cf. \cite[Proposition 6.10]{PlR}). The following
proposition provides an analog of this fact for the multinorm
principle.
\begin{proposition}\label{P:4}
Let $L_1$ and $L_2$ be two separable extensions of $K$ of a prime
degree $p$. Then the pair $L_1 , L_2$ satisfies the multinorm
principle.
\end{proposition}

(Note that in this proposition we don't need to assume that our
extensions or their Galois closures are linearly disjoint.)

\vskip2mm

\begin{lemma}\label{L:7}
Let $L_1$ and $L_2$ be finite extensions of $K$. For any finite
extension $P$ of $K$ of degree relatively prime to both $[L_1 : K]$
and $[L_2 : K]$, the validity of the multinorm principle for the
pair $L_1P , L_2P$ of extensions of $P$ implies its validity for the
pair $L_1 , L_2$.
\end{lemma}
\begin{proof}
For $i = 1, 2$, since $[L_i : K]$ is coprime to $[P : K]$, the
extensions $L_i$ and $P$ are linearly disjoint over $K$, which
implies that the norm map $N_{L_i/K}$ coincides (on $J_{L_i}$ and
$L^{\times}_i$ ) with the restriction of the norm map $N_{L_iP/P}$.
Now, suppose that the multinorm principle holds for the pair $L_1P ,
L_2P$ over $P$, and let
$$
x \in K^{\times} \cap N_{L_1/K}(J_{L_1}) N_{L_2/K}(J_{L_2}).
$$
Then it follows from the above remark that $x \in P^{\times} \cap
N_{L_1P/P}(J_{L_1P}) N_{L_2P/P}(J_{L_2P})$, and hence
$$
x = N_{L_1P/P}(y_1) N_{L_2P/P}(y_2) \ \ \text{for some} \ \ y_i \in
(L_iP)^{\times}, \ i = 1, 2.
$$
Applying $N_{P/K}$, we obtain
$$
x^{[P : K]} = N_{L_1/K}(N_{L_1P/L_1}(y_1))
N_{L_2/K}(N_{L_2P/L_2}(y_2)) \in N_{L_1/K}(L^{\times}_1)
N_{L_2/K}(L^{\times}_2).
$$
Since $x^{[L_1 : K]} \in N_{L_1/K}(L^{\times}_1)$ and the degrees
$[L_1 : K]$ and $[P : K]$ are relatively prime, we conclude that
$$
x \in N_{L_1/K}(L^{\times}_1) N_{L_2/K}(L^{\times}_2),
$$
proving the multinorm principle for $L_1 , L_2$.
\end{proof}

\vskip2mm

\noindent \textit{Proof of Proposition \ref{P:4}.} We first reduce
the proof to the case where both $L_1$ and $L_2$ are Galois
extensions of $K$. Let $E_1$ be the Galois closure of $L_1$ and let
$G = \mathrm{Gal}(E_1/K)$. Then $G$ is isomorphic to a subgroup of
the symmetric group $S_p$, so its Sylow $p$-subgroup $G_p$ is a
cyclic group of order $p$. Set $P = E^{G_p}_1$; then $E_1 = L_1P$.
Since the degree $[P : K]$ is coprime to $p$, according to Lemma
\ref{L:7}, it suffices to prove the multinorm principle for the pair
$L_1P , L_2P$ of extensions of $P$. This enables us to assume
without any loss of generality that one of the extensions is Galois.
Repeating the argument for the other extension, we can assume that
both extensions are Galois.

Now, let us consider the case where $L_1$ and $L_2$ are cyclic
Galois extensions of $K$ of degree $p$. By the Hasse theorem,
$L_i/K$ satisfies the norm principle for $i = 1, 2$. So, if $L_1
\cap L_2 = K$ then the multinorm principle for $L_1 , L_2$ follows
from Proposition \ref{P:2} as condition (2) therein obviously holds.
In the remaining case $L_1 = L_2$, the multinorm principle reduces
to the norm principle for $L_i$, and therefore holds as well. \hfill
$\Box$

\vskip2mm

\noindent \textbf{Remark.} If $L_1$ and $L_2$ are two separable
extensions of $K$ of a prime degree $p$, and $E_1$ and $E_2$ are
their Galois closures, then one of the following occurs: either the
degree of $E := E_1 \cap E_2$ is prime to $p$, or $E_1 = E_2$. To
see this, one first proves the following elementary lemma from group
theory: {\it Let $G$ be a transitive subgroup of $S_p$. If $N \neq
\{ 1 \}$ is a normal subgroup of $G$ then the order $\vert N \vert$
is divisible by $p$.} Then, if $E_1 \neq E_2$, for at least one $i
\in \{ 1 , 2 \}$, the group $\mathrm{Gal}(E_i/E)$ is a nontrivial
normal subgroup of $\mathrm{Gal}(E_i/K) \subset S_p$, hence has
order divisible by $p$. Since the order of $S_p$ is not divisible by
$p^2$, we obtain that $[E : K]$ is prime to $p$, as claimed.

Now, if $[E : K]$ is prime to $p$ then by Lemma \ref{L:7} it is
enough to prove the multinorm principle for the pair of extensions
$L'_1:= L_1E , L'_2 := L_2E$ of $E$. But the Galois closures of
$L'_1$ and $L'_2$ coincide with $E_1$ and $E_2$ respectively, hence
are linearly disjoint over $E$. So, the multinorm principle for
$L'_1 , L'_2$ immediately follows from Proposition \ref{P:2} as
$L'_1/E$ and $L'_2/E$ satisfy the norm principle.

An obvious way to construct distinct degree $p > 2$ extensions $L_1$
and $L_2$ of $K$ such that $E_1 = E_2$ is to pick an arbitrary
non-Galois degree $p$ extension $L_1$ and take for $L_2$ its
suitable Galois conjugate. We note, however, that the
group-theoretic constructions in \cite{Ito} allow one to produce
{\it non-conjugate} extensions with this property. In any case,
letting $P$ denote the fixed field of a Sylow $p$-subgroup of
$\mathrm{Gal}(E/K)$, we will have $L_1P = L_2P = E$. Then arguing as
in Lemma \ref{L:7} one shows that
$$
N_{L_1/K}(L^{\times}_1) = N_{L_2/K}(L^{\times}_2) \ \ \text{and} \ \
N_{L_1/K}(J_{L_1}) = N_{L_2/K}(J_{L_2})
$$
(even when $L_1$ and $L_2$ are not Galois conjugate!). Thus, in this
case the multinorm principle for $L_1 , L_2$ reduces to the norm
principle for $L_i/K$. This provides a somewhat more detailed
perspective on the result of Proposition \ref{P:4}.

\vskip3mm

Finally, we observe that the multinorm can be considered not only
for pairs but for any finite families of finite extensions of $K$.
More precisely, we say that a family $L_1, \ldots , L_m$ $(m
\geqslant 2)$ satisfies the multinorm principle if
$$
K^{\times} \cap N_{L_1/K}(J_{L_1}) \cdots N_{L_m/K}(J_{L_m}) =
N_{L_1/K}(L^{\times}_1) \cdots N_{L_m/K}(L^{\times}_m).
$$

\noindent \textbf{Example 2.} {\it The multinorm principle may fail
for a triple $L_1, L_2, L_3$ of finite Galois extensions of $K$ even
when the fields $L_i$ and $L_j$ are pairwise linearly disjoint over
$K$.} Indeed, set $K = \mathbb{Q}$ and
$$
L_1 = \mathbb{Q}(\sqrt{13}), \ L_2 = \mathbb{Q}(\sqrt{17}), \
\text{and} \ L_3 = \mathbb{Q}(\sqrt{13 \cdot 17}).
$$
Then
$$
K^{\times} \cap N_{L_1/K}(J_{L_1})
N_{L_2/K}(J_{L_2})N_{L_3/K}(J_{L_3}) = K^{\times},
$$
but
$N_{L_1/K}(L^{\times}_1)N_{L_2/K}(L^{\times}_2)N_{L_3/K}(L^{\times}_3)$
is a subgroup of $K^{\times}$ of index 2 (cf. \cite[Exercise
5]{Cass} and \cite[Lemma 4.8]{PR}), hence the multinorm principle
fails (see also \cite[\S 2]{Hurlimann}). \hfill $\Box$

\vskip2mm

Generalizing the Main Theorem of this note, one can show that given
finite Galois extensions $L_1, \ldots , L_m$ of $K$ such that
$$
\mathrm{Gal}(L_1 \cdots L_m/K) \simeq \mathrm{Gal}(L_1/K) \times
\cdots \times \mathrm{Gal}(L_m/K)
$$
(in other words, the whole family $L_1, \ldots , L_m$ is linearly
disjoint over $K$) then the multinorm principle still holds for
$L_1, \ldots , L_m$. This, however, requires some new considerations
which will be described in \cite{Pol}.

\vskip7mm

\centerline{\sc Appendix. Deflation and residuation maps and their
properties.}

\vskip3mm

In this appendix, we briefly sketch the construction of the
deflation and residuation maps and prove Lemma \ref{L:5} (note that
our account, unlike that in \cite{Weiss} and \cite{Horie}, is based
on homogeneous cochains).

Given a finite group $G$, we let $X = \{ X_i \}_{i \in \mathbb{Z}}$
denote the standard complex used to define the Tate cohomology
groups (cf. \cite[ch. IV, \S 6]{Cass}). More precisely, for $i
\geqslant 0$, $X_i = \mathbb{Z}[G^{i+1}]$ with the $G$-action
$s(g_0, \ldots , g_i) = (sg_0, \ldots , sg_i)$, and the differential
$d \colon X_{i+1} \to X_i$ given by
$$
d(g_0, \ldots , g_{i+1}) = \sum_{j = 0}^{i+1} (-1)^j (g_0, \ldots ,
g_{j-1} , g_{j+1}, \ldots , g_{i+1}).
$$
Furthermore, for $i \geqslant 1$, we set $X_{-i} =
\mathrm{Hom}_{\mathbb{Z}}(X_{i-1} , \mathbb{Z})$, which is a free
$\mathbb{Z}$-module with a basis $(s^*_1, \ldots , s^*_i)$, where
all $s_j \in G$, defined by
$$
(s^*_1, \ldots , s^*_i) (g_0, \ldots , g_{i-1}) = \left\{
\begin{array}{cl} 1 & \text{if} \ \ s_j = g_{j-1} \ \ \text{for all}
\  \ j, \\ 0 & \text{otherwise}, \end{array} \right.
$$
and the $G$-action $g(s^*_1, \ldots , s^*_i) = ((gs_1)^*, \ldots ,
(gs_i)^*)$. The differential $d \colon X_{-i} \to X_{-i-1}$ is given
by
$$
d(s^*_1, \ldots , s^*_i) = \sum_{j = 1}^{i+1} \sum_{g \in G} (-1)^j
(s^*_1, \ldots , s^*_{j-1}, g^*, s^*_j, \ldots , s^*_i).
$$
Finally, the ``special'' differential $d \colon X_0 \to X_{-1}$ is
defined by
$$
d(g_0) = \sum_{s \in G} s^*.
$$
Then for any $G$-module $A$ and all $i \in \mathbb{Z}$ we have
$$
\hat{H}^i(G , A) = H^i(\mathrm{Hom}_G(X , A)).
$$

\vskip2mm

\noindent \textit{Deflation map.} Given any normal subgroup $H$ of
$G$, we let $Y$ denote the standard complex for $G/H$.  Then for any
$G$-module $A$ and  each $i \geqslant 1$ there is a map $\delta_{-i}
\colon \mathrm{Hom}_G(X_{-i},A) \rightarrow
\mathrm{Hom}_G(Y_{-i},A)$ given by $$(\delta_{-i}f)(\alpha_1^*,
\dots, \alpha_i^*)= \sum_{g_iH = \alpha_i}f(g_1^*,\dots,g_i^*)$$ for
$f \in \mathrm{Hom}_G(X_{-i},A)$ and $\alpha_1, \dots, \alpha_i \in
G/H$. One can check that the image of $\delta_{-i}$ lies in
$\mathrm{Hom}_{G/H}(Y_{-i},A^H)$, hence $\delta_{-i}$ induces a map
\[ \mathrm{Def}^G_{G/H} \colon \hat{H}^{-i}(G,A) \rightarrow \hat{H}^{-i}(G/H,A^H) \]
called the {\it deflation map}.  For $i = 0$ one gives an ad hoc definition of the deflation map.
Namely, for any group $G$ and any $G$-module $A$ we have
$\hat{H}^0(G , A) \simeq A^G/N_G(A)$, where $N_G$ is the norm map,
$N_G(a) = \sum_{g \in G} ga$. Then
$$
\mathrm{Def}^G_{G/H} \colon \hat{H}^0(G , A) \to \hat{H}^0(G/H ,
A^H)
$$
is induced by the identification $A^G \to (A^H)^{G/H}$ and the
inclusion $N_G(A) \hookrightarrow N_{G/H}(A^H)$. (In terms of
homogeneous cochains, every element of $\hat{H}^0(G , A)$ is
represented by a function $f \in \mathrm{Hom}_G(\mathbb{Z}[G] , A)$
with values in $A^G$. Then $\mathrm{Def}^G_{G/H}$ is induced by the
map $\delta \colon \mathrm{Hom}_G(\mathbb{Z}[G] , A^G) \to
\mathrm{Hom}_{G/H}(\mathbb{Z}[G/H] , A^H)$ given by $\delta(f)(g_0H)
= f(g_0)$.)

\vskip2mm

\noindent {\it Residuation map.}  Let $G$, $H$, $X$, $Y$, and $A$ be
as above. We let $I_H$ denote the augmentation ideal of $\bZ[H]$,
and set $A_H = A/I_HA$. For each $i \geqslant 1$ there is a map
$\delta'_{-i} \colon \mathrm{Hom}_G(X_{-i},A) \rightarrow
\mathrm{Hom}_{G/H}(Y_{-i},A_H)$ given by
\[ (\delta'_if)(\alpha^*,\alpha_2^*, \dots, \alpha_i^*)= \sum_{g_iH = \alpha_i}f(g^*,g_2^*,\dots,g_i^*)+I_H, \]
where $g$ is an arbitrary (single) element such that $gH=\alpha$;
since $f$ is a $G$-map, this definition does not depend on the
choice of $g$. Then for $i \geqslant 2$, $\delta'_{-i}$ induces a
map on cohomology
\[\mathrm{Rsd}^G_{G/H} \colon \hat{H}^{-i}(G,A) \rightarrow \hat{H}^{-i}(G/H,A_H), \]
called the {\it residuation map}.
We note that in the special case where $A$ is a trivial $G$-module,
we have $A=A^H=A_H$, and
\begin{equation}\label{E:Rsd}
\vert H \vert \cdot \mathrm{Rsd}^G_{G/H} = \mathrm{Def}^G_{G/H}.
\end{equation}
We will make use of this fact below for $A = \bZ$.

\vskip3mm

\noindent {\it Proof of Lemma \ref{L:5}.} Fix $i \geqslant 2$, and
to simplify notation we will write $\mathrm{Inf}$, $\mathrm{Def}$, ...
instead of $\mathrm{Inf}^G_H$, $\mathrm{Def}^G_H$, etc. Let $\bar{f}
\in \hat{H}^{-i}(G , \mathbb{Z})$ and $\bar{\psi} \in \hat{H}^i(H ,
\mathbb{Z})$ be represented by the homogeneous cocycles $f \in
\mathrm{Hom}_G(\mathbb{Z}[(G^*)^i],\bZ)$, where $(G^*)^i = \{ (s^*_1,
\ldots , s^*_i) \, \vert \, s_j \in G \}$, and $\psi \in
\mathrm{Hom}_H(\mathbb{Z}[H^{i+1}] , \mathbb{Z})$. Furthermore,
$\mathrm{Def}(\bar{f})$ and $\mathrm{Rsd}(\bar{f})$ are represented
respectively by $\tilde{f}_1$ and $\tilde{f}_2  \in
\mathrm{Hom}_H(\mathbb{Z}[(H^*)^i] , \mathbb{Z})$ defined by
$$
\tilde{f}_1(h^*_1, \ldots , h^*_i) = \sum_{k_j \in K} f((h_1k_1)^*,
\ldots , (h_ik_i)^*)  \ \text{and}  \ \tilde{f}_2(h^*_1, h^*_2,
\ldots , h^*_i) = \sum_{k_j \in K} f(h^*_1, (h_2k_2)^*, \ldots ,
(h_ik_i)^*),
$$
and $\mathrm{Inf}(\bar{\psi})$ is represented by $\tilde{\psi} \in
\mathrm{Hom}_G(\mathbb{Z}[G^{i+1}] , \mathbb{Z})$ given by
$$
\tilde{\psi}(h_0k_0, \ldots , h_ik_i) = \psi(h_0, \ldots , h_i).
$$
Next, as shown in \cite[p. 105-108]{Cass}, the cup-product $\bar{a}
\cup \bar{b}$ of classes $\bar{a} \in \hat{H}^{-i}(G , \mathbb{Z})$
and $\bar{b} \in \hat{H}^i(G , \mathbb{Z})$ that are represented by
the cocycles $a$ and $b$, is represented by the function
$$
g_0 \mapsto \sum_{s_1, \ldots , s_i \in G} a(s^*_1, \ldots , s^*_i)
b(s_i, \ldots , s_1, g_0),
$$
and the cup-product of classes in $\hat{H}^{-i}(H , \mathbb{Z})$ and
$\hat{H}^i(H , \mathbb{Z})$ is described similarly. Finally, the
corestriction map from $H^0(H , \mathbb{Z}) = \mathbb{Z}/\vert H
\vert\mathbb{Z}$ to $H^0(G , \mathbb{Z}) = \mathbb{Z}/\vert G \vert
\mathbb{Z}$ is given by multiplication by $[G : H] = \vert K \vert$.

Putting this information together, we obtain that
$\mathrm{Cor}(\mathrm{Rsd}(\bar{f}) \cup \bar{\psi})$ is represented
by the function
$$
h_0k_0 \mapsto \vert K \vert \sum_{h_1, \ldots , h_i \in H}
\tilde{f}_2(h^*_1, \ldots , h^*_i) \psi(h_i, \ldots , h_1, h_0),
$$
and therefore in view of (\ref{E:Rsd}) by the function
$$
h_0k_0 \mapsto \sum_{h_1, \ldots , h_i \in H} \tilde{f}_1(h^*_1,
\ldots , h^*_i) \psi(h_i, \ldots , h_1, h_0) = \sum_{h_j \in H}
\sum_{k_j \in K} f((h_1k_1)^*, \ldots , (h_ik_i)^*) \psi(h_i, \ldots
, h_1, h_0)
$$
$$
= \sum_{h_j \in H, \, k_j \in K} f((h_1 k_1)^*, \ldots , (h_ik_i)^*)
\tilde{\psi}(h_1k_1, \ldots , h_ik_i, h_0k_0) = \sum_{s_j \in G}
f(s^*_1, \ldots , s^*_i) \tilde{\psi}(s_i, \ldots , s_1, h_0k_0).
$$
But the function
$$
h_0k_0 \mapsto \sum_{s_j \in G} f(s^*_1, \ldots , s^*_i)
\tilde{\psi}(s_i, \ldots , s_1, h_0k_0)
$$
also represents $\bar{f} \cup \mathrm{Inf}(\bar{\psi})$, yielding
our claim. \hfill $\Box$

\vskip5mm

\noindent {\small {\textbf{Acknowledgements.} The second-named
author was partially supported by NSF grant DMS-0965758, BSF grant
2010149 and the Humboldt Foundation. During the preparation of the
final version of this paper, he  was visiting the Mathematics
Department of the University of Michigan as a Gehring Professor; the
hospitality and generous support of this institution are thankfully
acknowledged. }}

\vskip5mm

\bibliographystyle{amsplain}

\end{document}